\pdfoutput=1
\documentclass[10pt,a4paper,oneside]{amsart}
\usepackage{amsmath, amssymb, amsfonts}
\usepackage[a4paper]{geometry}

\numberwithin{equation}{section}

\usepackage{stmaryrd}
\newcommand{\llb}{\llbracket}
\newcommand{\rrb}{\rrbracket}

\renewcommand{\epsilon}{\varepsilon}
\renewcommand{\phi}{\varphi}
\renewcommand{\theta}{\vartheta}

\newcommand{\cent}{\mathfrak z}

\DeclareMathOperator{\Aut}{Aut}

\DeclareMathOperator{\Gal}{Gal}
\DeclareMathOperator{\GL}{GL}
\DeclareMathOperator{\id}{id}
\DeclareMathOperator{\ind}{ind}
\DeclareMathOperator{\Irr}{Irr}
\DeclareMathOperator{\res}{res}
\DeclareMathOperator{\Quot}{Quot}

\newcommand{\G}{\mathcal G}
\newcommand{\OO}{\mathcal O}
\newcommand{\Q}{\mathcal Q}

\newcommand{\QQ}{\mathbb Q}
\newcommand{\ZZ}{\mathbb Z}

\newcommand{\pt}{\delta_{\tauf}}
\newcommand{\ptq}{\delta_{\tau_W}}
\newcommand{\gcpf}{\gamma^{\prime\prime}_{F,\eta}}
\newcommand{\Gcpf}{\Gamma^{\prime\prime}_{F,\eta}}
\newcommand{\tD}{\widetilde D}
\newcommand{\tauf}{\tau_F}
\newcommand{\vf}{v_{F,\chi}}
\newcommand{\DFchi}{D_{F,\chi}}
\newcommand{\DWchi}{D_{W,\chi}}
\newcommand{\nFchi}{n_{F,\chi}}
\newcommand{\sFchi}{s_{F,\chi}}
\newcommand{\ff}{\mathbf f}
\newcommand{\epsilonFchi}{\epsilon_{F,\chi}}
\newcommand{\fetaj}{f_{F,\eta}^{(j)}}
\newcommand{\fWjetai}{f^{(j)}_{W,\eta_i}}
\newcommand{\ffFjetai}{\ff^{(j)}_{F,\eta_i}}
\newcommand{\ffFjeta}{\ff^{(j)}_{F,\eta}}
\newcommand{\ffFchijeta}{\ff^{(j)}_{F_\chi,\eta}}
\newcommand{\fWjeta}{f^{(j)}_{W,\eta}}
\newcommand{\Achi}{A_{F,\chi}}
\newcommand{\AAchi}{\mathcal A_{F,\chi}}

\usepackage{colonequals}

\usepackage{tikz}
\usetikzlibrary{cd}
\usetikzlibrary{arrows}
\usetikzlibrary{positioning}
\usetikzlibrary{calc}

\usepackage[style=alphabetic,backend=biber,isbn=false]{biblatex}
\usepackage{enumitem}
\newlist{theoremlist}{enumerate}{1}
\setlist[theoremlist]{label=(\roman{theoremlisti}), ref=\thetheorem.\roman{theoremlisti},noitemsep, topsep=.2ex}
\newlist{propositionlist}{enumerate}{1}
\setlist[propositionlist]{label={(\roman{propositionlisti})}, ref=\theproposition.\roman{propositionlisti},noitemsep, topsep=.2ex}
\newlist{lemmalist}{enumerate}{1}
\setlist[lemmalist]{label=(\roman{lemmalisti}), ref=\thelemma.\roman{lemmalisti},noitemsep, topsep=.2ex}

\usepackage{etoolbox} 

\usepackage{mathtools}

\usepackage[hidelinks,plainpages=false,pdfpagelabels]{hyperref}
\usepackage{thmtools} 
\usepackage[capitalise, noabbrev]{cleveref}

\theoremstyle{plain}
\newtheorem{theorem}{Theorem}[section]
\newtheorem{lemma}[theorem]{Lemma}
\newtheorem{proposition}[theorem]{Proposition}
\newtheorem{corollary}[theorem]{Corollary}

\theoremstyle{definition}

\theoremstyle{remark}
\newtheorem{remark}[theorem]{Remark}

\Crefname{lemma}{Lemma}{Lemmata}
\Crefname{proposition}{Proposition}{Propositions}
\Crefname{conjecture}{Conjecture}{Conjectures}
\Crefname{example}{Example}{Examples}
\Crefname{theoremlisti}{Theorem}{Theorems}
\Crefname{propositionlisti}{Proposition}{Propositions}
\Crefname{conjecturelisti}{Conjecture}{Conjectures}
\Crefname{lemmalisti}{Lemma}{Lemmata}

\title[On the Wedderburn decomposition of $\Q^ F(\G)$, part II]{On the Wedderburn decomposition of the total ring of quotients of certain Iwasawa algebras II}
\author{Ben Forrás}
\address{University of Ottawa\\
	Department of Mathematics and Statistics \\
    STEM Complex \\
	150 Louis-Pasteur Pvt\\
    Ottawa, ON \\
	Canada K1N 6N5}
\email{bforras@uottawa.ca}
\urladdr{https://bforras.eu}

\date{Version of 2026-01-27}

\begin{document}

\begin{abstract}
Let $\G\simeq H\rtimes\Gamma$ be the semidirect product of a finite group $H$ and $\Gamma\simeq\ZZ_p$. Let $ F/\QQ_p$ be a finite extension with ring of integers $\OO_ F$. Then the total ring of quotients $\Q^ F(\G)$ of the completed group ring $\OO_{ F}\llb\G\rrb$ is semisimple artinian. 
We determine its Wedderburn decomposition in full generality in terms of the Wedderburn decomposition of the group ring $ F[H]$. 
Such a description was previously available only for those simple components for which a certain associated field extension is totally ramified.
\end{abstract}

\maketitle

\section{Introduction} \label{sec:introduction}
\subsection{Overview} \label{sec:overview}
Let $p$ be an odd prime number.
Let $\G= H\rtimes \Gamma$ be a profinite group that can be written as the semidirect product of a finite normal subgroup $H$ and a pro-$p$ group $\Gamma\simeq\ZZ_p$ isomorphic to the additive group of the $p$-adic integers. Let $ F$ be a finite extension of $\QQ_p$ with ring of integers $\OO_ F$, and consider the completed group ring
$\Lambda^{\OO_ F}(\G)\colonequals \OO_ F\llb \G\rrb$.
For $n_0$ sufficiently large, the subgroup $\Gamma_0\colonequals \Gamma^{p^{n_0}}$ is central in $\G$. 
Let $\Q^ F(\G)\colonequals \Quot(\Lambda^{\OO_ F}(\G))$ be the total ring of quotients of $\Lambda^{\OO_ F}(\G)$.

The ring $\Q^{ F}(\G)$ is a semisimple artinian $\Lambda^{\OO_F}(\Gamma_0)$-algebra, as shown by Ritter and Weiss \cite{TEIT-II}. As a semisimple algebra, it has a Wedderburn decomposition
\begin{equation} \label{eq:QFG-Wedderburn}
    \Q^ F(\G) \simeq \bigoplus_{\chi\in \Irr(\G)/\sim_ F} M_{\nFchi}(\DFchi)  .
\end{equation}
Here $\Irr(\G)$ denotes the set of irreducible characters of $\G$ with open kernel, and the equivalence relation $\sim_ F$ on $\Irr(\G)$ is defined as follows: two characters $\chi,\chi'$ are equivalent if there is a $\sigma\in\Gal( F_\chi/ F)$ such that ${}^\sigma(\res^\G_H \chi)=\res^\G_H \chi'$, where $ F_\chi\colonequals F(\chi(h):h\in H)$. For each equivalence class, we have a skew field $\DFchi$ of Schur index $\sFchi$.

The aim of this paper is providing a completely general description of the Wedderburn decomposition of $\Q^F(\G)$. In previous work of the author \cite{W}, such a description was achieved under an additional ramification hypothesis (see \cref{DchiW-bigoplus} for a precise statement). 
Prior to that work, the Wedderburn decomposition was only known in the cases when $\G$ is either a pro-$p$ group \cite{Lau}, or the $\Gamma$-action on $H$ is trivial (in which case $\Q^F(\G)=\Q^F(\Gamma)\otimes_F F[H]$). Partial results towards a general description were obtained by Nickel \cite{NickelConductor}. For a more thorough overview of these results, we refer to the Introduction of \cite{W}. 

An earlier version of the results contained in this work was originally part of a preprint version of \cite{W}, but these were subsequently reorganized into a separate manuscript. The relevant notions and results from that article shall be recalled below.

\subsection{Statement of results} \label{sec:statement-of-results}
Consider the group algebra $F[H]$, which is semisimple by Maschke's theorem. Therefore it admits a Wedderburn decomposition
\begin{equation} \label{eq:QpH-Wedderburn}
    F[H] \simeq \bigoplus_{\eta\in\Irr(H)/\sim_F} M_{n_\eta}(D_\eta).
\end{equation}
As before, two irreducible characters $\eta,\eta'\in\Irr(H)$ are equivalent if there exists some $\sigma\in\Gal(F(\eta)/F)$ such that ${}^\sigma\eta=\eta'$, where $F(\eta)\colonequals F(\eta(h):h\in H)$. On the right hand side, each $D_\eta$ is a finite dimensional skew field over its centre $F(\eta)$. In particular, these are finite dimensional skew fields over local fields and hence well understood using Hasse's theory, see \cite{Hasse} or \cite[\S31]{MO}. We shall describe the Wedderburn decomposition \eqref{eq:QFG-Wedderburn} in terms of the Wedderburn decomposition \eqref{eq:QpH-Wedderburn}.

Let us fix a topological generator $\gamma$ of $\Gamma$. We let $\Gamma$ and $\Gal( F(\eta)/ F)$ act on $\Irr(H)$: for $\eta\in\Irr(H)$, let ${}^{\gamma}\eta(h)=\eta(\gamma h\gamma^{-1})$ and ${}^\sigma\eta(h)=\sigma(\eta(h))$ for all $h\in H$ and $\sigma\in \Gal( F(\eta)/ F)$.
These two actions are related as follows: for $\chi\in\Irr(\G)$ and $\eta\mid\res^\G_H\chi$ an irreducible constituent of its restriction to $H$, we define $\vf$ to be the minimal positive exponent such that $\gamma^{\vf}$ acts as a Galois automorphism of $ F(\eta)/ F$ on $\eta$; it can be shown that $\vf$ depends only on $\chi$ and $F$. There is a unique automorphism $\tauf\in \Gal( F(\eta)/ F)$ such that 
\[{}^{\gamma^{\vf}}\eta={}^{\tauf}\eta  .\] 
Furthermore, this $\tauf$ fixes $ F_\chi$, {and} generates $\Gal( F(\eta)/ F_{\chi})$. Moreover, $\tauf$ admits a unique extension as an automorphism of $D_\eta$ of the same order, which we will also denote by $\tauf$.

Our main result is the following: 
\begin{theorem} \label{thm:main}
	Let $ F$ be a finite extension of $\QQ_p$, let $\chi\in\Irr(\G)$, and let $\eta\mid\res^\G_H\chi$ be an irreducible constituent. Then for the $\chi$-component in the Wedderburn decomposition \eqref{eq:QFG-Wedderburn},
	\begin{theoremlist}
		\item the size of the matrix ring is $\nFchi = n_\eta \vf$, \label{item:main-n}
		\item the Schur index of $\DFchi$ is $\sFchi = s_\eta (F(\eta):F_\chi)$, \label{item:main-s}
		\item the skew field is isomorphic to the total ring of quotients of a skew power series ring: $\DFchi \simeq \Quot\left(\OO_{D_\eta}[[X; \tauf, \tauf-\id]]\right)$. \label{item:main-D}
	\end{theoremlist}
\end{theorem}
The ring {$\OO_{D_\eta}[[X; \tauf, \tauf-\id]]$} is the skew power series ring whose underlying additive group agrees with that of the power series ring $\OO_{D_\eta}[[X]]$, with multiplication rule {$Xd=\tauf(d)X+(\tauf-\id)(d)$} for all $d\in\OO_{D_\eta}$. Such skew power series rings were studied by Schneider and Venjakob \cite{VenjakobWPT,SchneiderVenjakob} in general. The kind of skew power series rings occurring in \cref{item:main-D} were investigated in \cite[\S3]{W}.

\subsection{Applications} The description of the skew fields in terms of skew power series allows one to describe all maximal orders in $\Q^F(\G)$, which will be done in \cref{cor:max-orders}. This has applications in non-commutative Iwasawa theory. Indeed, employing Venjakob's Weierstraß theory for skew power series rings, it is possible to verify integrality of certain characteristic elements attached to number fields \cite{epac} or to elliptic curves \cite{FM}.

\subsection{Outline} We begin by recalling the relevant results of \cite{W} in \cref{sec:preliminaries}. In \cref{sec:indecomposables-QF}, we construct a system of indecomposable idempotents in $\Q^F(\G)$, which are then used in \cref{sec:description} to obtain an explicit description of the skew fields $\DFchi$. A set of examples in which the extension $F(\eta)/F_\chi$ is unramified is constructed in \cref{sec:unramified}. The case of $p=2$ is discussed in \cref{sec:2adic}.

\subsection{Notation and conventions} The number $p$ is an odd rational prime, except in \cref{sec:2adic}. We will abuse notation by writing $\oplus$ for a direct product of rings, even though this is not a coproduct in the category of rings. We denote the centre of a ring $R$ by $\cent(R)$. The total ring of quotients $\Quot(R)$ of a ring $R$ is obtained from $R$ by inverting all central regular elements.

Many objects depend on the base field $F$, and several proofs involve a base change. This necessitates carrying the base field around in our notation as a subscript. Note that in \cite{W}, the base field appears in the superscript instead; aside from that, our notation is entirely congruous with the one used in that work.

\subsection{Acknowledgements} 
The author thanks Andreas Nickel for his advice and numerous helpful discussions. He also thanks the anonymous referee of \cite{W} for pointing out mistakes in an earlier version of the proof and for several constructive comments.

This work was partially funded by the Deutsche Forschungsgemeinschaft (DFG, German
Research Foundation) under project number 559516518.

\section{Preliminaries} \label{sec:preliminaries}
In \cite{W}, \cref{thm:main} was established under the condition that $F(\eta)/F_\chi$ is totally ramified. This is not always the case: see \cref{sec:unramified} for an example. In this section, for the convenience of the reader, we briefly recall some of the results of \cite{W} which will be used later on. We retain the notation introduced in \cref{sec:introduction}. In particular, $\chi\in\Irr(\G)$ is an irreducible character of $\G$ with open kernel, and $\eta\mid\res^\G_H\chi$ is an irreducible constituent of its restriction to $H$.

\subsection{Galois action on irreducible characters}
Let $w_\chi$ be the index of the stabiliser of $\eta$ in $\Gamma$; this depends only on $\chi$ and not on $\eta$ or $F$.
Let $\vf$ be as in \cref{sec:statement-of-results}, that is,
\begin{equation} \label{eq:def-v-chi}
    \vf=\min\left\{0<i : \exists \tau\in\Gal(F(\eta)/F), {}^{\gamma^{i}}\eta={}^\tau \eta\right\}.
\end{equation}
It is clear that $\vf\mid w_\chi$.
\begin{proposition}[{\cite[\S4.2]{W}}]
    Let $\chi\in\Irr(\G)$ and let $\eta\mid\res^\G_H\chi$ be an irreducible constituent thereof. Then the following hold.
    \begin{propositionlist}
        \item The number $\vf$ depends only on $F$ and $\chi$ and not on $\eta$.
        \item The extension $\Gal(F(\eta)/F_\chi)$ is cyclic of order $w_\chi/\vf$. In particular, $(F(\eta):F_\chi)=w_\chi/\vf$. \label{item:Feta-Fchi-degree}
        \item There is a unique automorphism $\tau=\tauf\in\Gal(F(\eta)/F)$ as in \eqref{eq:def-v-chi}, that is, satisfying ${}^{\gamma^{i}}\eta={}^{\tauf} \eta$. This $\tauf$ fixes $F_\chi$ and generates $\Gal(F(\eta)/F_\chi)$.
        \item There is a collection of $\vf$ many irreducible constituents $\eta_1,\ldots,\eta_{\vf}\in\Irr(H)$ of $\res^\G_H\chi$ such that the sum over their Galois orbits recovers $\res^\G_H\chi$, that is, such that
        \begin{equation} \label{eq:res-eta-i}
            \res^\G_H \chi = \sum_{i=1}^{\vf} \sum_{\sigma\in\Gal(F(\eta)/F_\chi)} {}^\sigma \eta_i.
        \end{equation} \label{item:res-G-H-E}
        \item For any intermediate field extension $F_\chi\subseteq E\subseteq F(\eta)$, we have $\tau^E=(\tau^F)^{(E:F_\chi)}$ and $v_\chi^E=v_\chi^F{(E:F_\chi)}$. \label{item:change-tau}
    \end{propositionlist}
\end{proposition}
Note that the field $F(\eta)$ is the same for each irreducible constituent $\eta$ of $\res^\G_H\chi$, so the right hand side of \eqref{eq:res-eta-i} above makes sense.

Consider the $\eta$-component of the Wedderburn decomposition \eqref{eq:QpH-Wedderburn} of $F[H]$. The skew field part of this is $D_\eta$, which is a finite dimensional skew field over the local field $F(\eta)$. By Hasse's theory \cite[\S14]{MO}, there are elements $\omega\in D_\eta$ and $\pi\in D_\eta$ such that $F(\eta)(\omega)$ resp. $F(\eta)(\pi)$ are unramified resp. totally ramified extensions of $F(\eta)$, and both are maximal subfields of $D_\eta$. Then the following holds; see also the end of \S4.2 of \cite{W}.
\begin{proposition}[{\cite[Proposition~2.7]{W}}] \label{extending-tau}
	The automorphism $\tau\in\Gal(F(\eta)/F_\chi)$ admits a unique extension to $D_\eta$ of order $w_\chi/\vf$ such that $\tau(F(\eta)(\omega))=F(\eta)(\omega)$ and $\tau(F(\eta)(\pi))=F(\eta)(\pi)$.
\end{proposition}

\subsection{Primitive central idempotents}
Recall the definitions of the following idempotents employed by Ritter--Weiss \cite{TEIT-II} and Nickel \cite{NickelConductor}. For $\chi$ and $\eta$ as above, define
\begin{align*}
	e(\eta) &\colonequals \frac{\eta(1)}{\#H} \sum_{h\in H} \eta(h^{-1}) h \in  \QQ_p(\eta)[H] \subseteq F(\eta)[H]  , &
	e_{\chi} &\colonequals \sum_{g\in\G/\G_\eta} e({}^g\eta) \in \QQ_p(\eta)[H] \subseteq F(\eta)[H] .
\end{align*}
These are primitive central idempotents in $F(\eta)[H]$ and $\Q^{F(\eta)}(\G)$, respectively; see \cite[Proposition~6(2)]{TEIT-II} for the latter claim. As we wish to work with coefficients in $F$ rather than in $F(\eta)$, we consider the following modifications:
\begin{align} \label{eq:def-epsilon}
	\epsilon_F(\eta) &\colonequals \sum_{\sigma\in\Gal( F(\eta)/ F)} e({}^\sigma\eta) \in  F[H], &
	\epsilon_{F,\chi} &\colonequals \sum_{\sigma\in\Gal( F_{\chi}/ F)} \sigma(e_{\chi}) \in  F[H].
\end{align}
These are primitive central idempotents in $F[H]$ and $\Q^F(\G)$, respectively. Therefore $F[H]\epsilon_F(\eta)\simeq M_{n_\eta}(D_\eta)$ and $\Q^F(\G)\epsilon_{F,\chi}\simeq M_{n_\chi}(D_\chi)$.

As a consequence of \cref{item:res-G-H-E} (see also \cite[Lemma~1.1]{NickelConductor}), we obtain the following:
\begin{align} \label{eq:epsilon-chi-etas}
	\epsilon_{F,\chi} &= {\sum_{i=1}^{\vf}} \epsilon_F(\eta_i)  .
\end{align}

\subsection{Indecomposable idempotents} \label{sec:indecomposables}
Fix a Wedderburn isomorphism as in \eqref{eq:QpH-Wedderburn}. Let $\fetaj\in F[H]$ be the element corresponding to the $n_\eta\times n_\eta$-matrix in the $\eta$-component on the right hand side of \eqref{eq:QpH-Wedderburn} that consists entirely of zeros except for a single $1$ in the $j$th entry of the diagonal. 
Note that $\fetaj$ is an indecomposable idempotent in $F[H]$, and $\fetaj F[H] \fetaj \simeq D_\eta$. The primitive central idempotent $\epsilon_F(\eta)$ is the sum of these indecomposable idempotents:
\[\epsilon_F(\eta)=\sum_{j=1}^{n_\eta} \fetaj.\]

\begin{theorem}[{\cite[Theorem~6.1 and Proposition~6.10]{W}}] \label{thm:indecomposability-W}
    Let $\chi\in\Irr(\G)$ and let $\eta\mid\res^\G_H\chi$ be an irreducible constituent thereof. Suppose that the extension $F(\eta)/F_\chi$ is \textbf{totally ramified}. Then $\fetaj$ is an indecomposable idempotent in $\Q^F(\G)$, and $\fetaj \Q^F(\G) \fetaj \simeq D_\chi$.
\end{theorem}

We recall the first step in the proof of \cref{thm:indecomposability-W}: this is reformulating the claim in terms of linear algebra. 
Let $x\in \Q^F(\G)$ be an element such that $\fetaj x \fetaj \ne0$. \Cref{thm:indecomposability-W} states that there exists a $y\in\Q^F(\G)$ such that $\fetaj x \fetaj \cdot \fetaj y \fetaj =1$.

Recall that $n_0$ is large enough such that $\Gamma_0=\Gamma^{p^{n_0}}$ is central in $\G$. Then there is a decomposition of left $\Q^F(\Gamma_0)$-modules
\begin{equation} \label{eq:QFG-dec}
    \Q^F(\G)=\bigoplus_{i=0}^{p^{n_0}-1} \Q^F(\Gamma_0)[H] \gamma^i,
\end{equation}
Let $\tD_\eta\colonequals D_\eta\otimes_{F(\eta)} \Q^{F(\eta)}(\Gamma_0)$. Then the isomorphism $F[H]\epsilon_F(\eta)\simeq M_{n_\eta}(D_\eta)$ extends to an isomorphism
$\Q^F(\G)\epsilon_F(\eta)\simeq M_{n_\eta}(\tD_\eta)$. The automorphism $\tauf$ extends to $\tD_\eta$ by letting it act trivially on the topological generator $\gamma_0\colonequals\gamma^{p^{n_0}}$ of $\Gamma_0$. This defines an entry-wise action of $\tauf$ on the matrix ring $M_{n_\eta}(\tD_\eta)$. Let $\pt\in\Aut(\Q^F(\G)\epsilon_F(\eta))$ be the pullback of this automorphism under $\Q^F(\G)\epsilon_F(\eta)\simeq M_{n_\eta}(\tD_\eta)$.

Under the decomposition \eqref{eq:QFG-dec}, we may write $x=\sum_{k=0}^{p^{n_0}-1} x_{k} \gamma^{k}$ and $y=\sum_{\ell=0}^{p^{n_0}-1} y_\ell \gamma^\ell$ with $x_{k},y_\ell\in\Q^F(\Gamma_0)[H]$. We now describe the rules of computation in $\fetaj \Q^F(\G) \fetaj$.
\begin{proposition}[{\cite[\S5]{W}}] \label{recall-multiplication}
    Let $F/\QQ_p$ be any finite extension and let $x,y\in \Q^F(\G)$ be arbitrary.
    \begin{propositionlist}
        \item \label{item:f-vf-div} We have $\fetaj x_{k} \fetaj=0$ whenever $\vf\nmid k$.
        \item \label{item:multiplication-rule} There exists a collection of elements $a_{F,i}\in F[H]\epsilon_F(\eta)^\times$ for $i\ge0$ such that $a_{F,i} \pt^{i}\left(a_{F,j}\right)=a_{F,i+j}$ for all $i,j\ge0$, and $\fetaj x \fetaj \cdot \fetaj y \fetaj$ equals
        \[ \sum_{\substack{k,\ell=0 \\ \vf\mid k,\ell}}^{p^{n_0}-1} 
	f_{\eta}^{(j)} x_k  a_{F,k/\vf} f_{\eta}^{(j)} \cdot 
	f_{\eta}^{(j)} \pt^{k/\vf}(y_{\ell}) \pt^{k/\vf}(a_{F,\ell/\vf}) f_{\eta}^{(j)} 
	\cdot a_{F,(k+\ell)/\vf}^{-1} \cdot \gamma^{k+\ell}.\]
        Furthermore, $a_{F,p^{n_0}/v_\chi}$ is central in $\Q^F(\Gamma_0)[H]\epsilon_F(\eta)$.
        \item Let $\gcpf\colonequals a_{F,1}^{-1} \gamma^{\vf}\in\Q^F(\G)$. Then $(\gcpf)^j=a_{F,j}^{-1} \gamma^{j\vf}$ holds for all $j\ge0$. Furthermore, conjugation by $\gcpf$ acts as $\pt$ on $\fetaj \Q^F(\Gamma_0)[H] \fetaj$, that is, for all $x_0\in\Q^F(\Gamma_0)[H]$:
        \[\gcpf \cdot \fetaj x_0 \fetaj (\gcpf)^{-1} = \fetaj \pt(x_0) \fetaj.\]
        \item For any intermediate field extension $F_\chi\subseteq E\subseteq F(\eta)$, we have $\delta_{\tau_E}=\pt^{(E:F_\chi)}$ and $\gamma_{E,\eta}^{\prime\prime}=\left(\gcpf\right)^{(E:F_\chi)}$. \label{item:changing-gcp}
    \end{propositionlist}
\end{proposition}

Finally, we record a description of $\DFchi$ in terms of $\gcpf$ under the total ramification assumption.
\begin{proposition}[{\cite[Corollary 6.11]{W}}] \label{DchiW-bigoplus}
    Let $\chi\in\Irr(\G)$ and let $\eta\mid\res^\G_H\chi$ be an irreducible constituent thereof. Suppose that the extension $F(\eta)/F_\chi$ is \textbf{totally ramified}. Then there is an isomorphism of rings
    \[\DFchi \simeq \bigoplus_{\substack{\ell=0\\ \vf\mid \ell}}^{p^{n_0}-1} \tD_\eta \left(\gcpf\right)^{\ell/\vf}\]
    where conjugation by $\gcpf$ acts as $\tauf$ on $\tD_\eta$. Moreover, the statements of \cref{thm:main} hold in this case.
\end{proposition}

\section{Indecomposable orthogonal idempotents in \texorpdfstring{$\Q^F(\G)$}{Q\^F(G)}} \label{sec:indecomposables-QF}
We will construct indecomposable idempotents in $\Q^F(\G)$ by reducing to the totally ramified case as follows. If $W$ is the maximal unramified extension of $F_\chi$ inside $F(\eta)$, then as we have seen in \cref{thm:indecomposability-W}, the idempotents $\fWjetai$ are indecomposable in $\Q^W(\G)$. We will show that taking sums over Galois orbits with respect to $W/F$ gives rise to indecomposable idempotents $\ffFjetai$ in $\Q^F(\G)$.

\subsection{Unramified base change}
Let us fix a character $\chi\in\Irr(\G)$, and let $\eta\mid\res^\G_H\chi$ be an irreducible constituent. We write $e$ resp. $f$ for the ramification index resp. inertia degree of $ F(\eta)/ F_\chi$.
The inertia group of $F(\eta)/F_\chi$ is the cyclic group $\langle\tauf^f\rangle$. Its fixed field $W\colonequals F(\eta)^{\langle \tauf^f\rangle}$ is the maximal unramified extension of $ F_\chi$ inside $ F(\eta)$: this is the unique unramified extension of $F_\chi$ of degree $f$. Notice that $W(\eta)= F(\eta)$ and $W_\chi=W$; in particular, the extension $W(\eta)/W_\chi$ is totally ramified, and so \cref{thm:indecomposability-W} applies to $W$.

Recall that we fixed a Wedderburn decomposition isomorphism of $F[H]$ as in \eqref{eq:QpH-Wedderburn}. For any intermediate Galois extension $F\subseteq E\subseteq F(\eta)$, this induces a Wedderburn isomorphism for $E[H]$. In other words, the following diagram is commutative:
\[\begin{tikzcd}
	& {F[H]} & {\displaystyle\bigoplus_{\eta\in\Irr(H)/\sim_F} M_{n_\eta}(D_\eta)} \\
	{E\otimes_F F[H]} & {E[H]} & {\displaystyle\bigoplus_{\eta\in\Irr(H)/\sim_E} M_{n_\eta}(D_\eta)} & {\displaystyle\bigoplus_{\eta\in\Irr(H)/\sim_F} \bigoplus_{\Gal(E/F)} M_{n_\eta}(D_\eta)}
	\arrow["\simeq", from=1-2, to=1-3]
	\arrow[hook', from=1-2, to=2-2]
	\arrow[hook', from=1-3, to=2-3]
	\arrow["\simeq"{description,font=\normalsize}, draw=none, from=2-1, to=2-2]
    \arrow["\simeq"{description,font=\normalsize}, draw=none, from=2-3, to=2-4]
	\arrow["\simeq", from=2-2, to=2-3]
\end{tikzcd}\]
In particular, the same matrix rings occur in all of these Wedderburn decompositions, but with different multiplicities. Recall that we defined the indecomposable idempotents $\fetaj$ in terms of the fixed Wedderburn isomorphism of $F[H]$. Similarly, we have idempotents $f^{(j)}_{E,\eta}$ defined via the lower row of the diagram. Since the diagram is commutative, in the group ring $E[H]$, we may identify the image of $\fetaj$ under the natural embedding $F[H]\hookrightarrow E[H]$ with $f^{(j)}_{E,\eta}$.

Using \eqref{eq:def-epsilon} and \eqref{eq:epsilon-chi-etas}, we have the following decompositions:
\begin{align} 
    \epsilonFchi &= \sum_{i=1}^{\vf}\sum_{\psi\in\Gal(W/F)} \psi\left(\epsilon_W(\eta_i)\right) \label{eq:epsilon-chi-E} \\
    &= \sum_{i=1}^{\vf}\sum_{j=1}^{n_\eta}\sum_{\psi\in\Gal(W/F)} \psi\left(\fWjetai\right). \label{eq:epsilon-chi-E-f}
\end{align}

\begin{lemma} \label{sigma-epsilon-L-j-i-orthogonal}
    The idempotents $\psi(\epsilon_W(\eta_i))$ in \eqref{eq:epsilon-chi-E} orthogonal, that is, for all $\psi,\psi'\in\Gal(W/F)$ and $1\le i,i'\le \vf$,
    \[\psi\left(\epsilon_W(\eta_i)\right)\cdot \psi'\left(\epsilon_W(\eta_{i'})\right)=\delta_{\psi\psi'}\delta_{ii'} \psi\left(\epsilon_W(\eta_i)\right).\]
\end{lemma}
\begin{proof}
    Applying \eqref{eq:def-epsilon}, the product on the left hand side becomes
    \begin{align}
        &\phantom=\psi\left(\sum_{\rho\in\Gal(F(\eta)/W)} e\left({}^\rho \eta_i\right)\right) \cdot \psi'\left(\sum_{\rho'\in\Gal(F(\eta)/W)} e\left({}^{\rho'} \eta_{i'}\right)\right) = \notag \\
        &=\sum_{\rho,\rho'\in\Gal(F(\eta)/W)} \hat\psi \rho \left(e\left(\eta_i\right)\right) \cdot \hat\psi' \rho' \left(e\left(\eta_{i'}\right)\right), \label{eq:rho-double-sum}
    \end{align}
    where $\hat\psi$ resp. $\hat\psi'$ are arbitrary lifts of $\psi$ resp. $\psi'$ under the natural map $\Gal(F(\eta)/F)\twoheadrightarrow\Gal(W/F)$. We assume $\hat\psi=\hat\psi'$ if $\psi=\psi'$.

    If $i\ne i'$, then $\eta_i$ and $\eta_{i'}$ are in separate $\Gal(F(\eta)/W)$-orbits (see \cref{item:res-G-H-E}), hence all terms in \eqref{eq:rho-double-sum} are zero.
    
    Suppose that $i=i'$. If $\psi=\psi'$, then the right hand side is clearly $\psi(\epsilon_W(\eta_i))$. Conversely, suppose that the double sum \eqref{eq:rho-double-sum} is nonzero.
    Since $i=i'$, the idempotents in \eqref{eq:rho-double-sum} are orthogonal, so their product is nonzero iff the respective characters agree, which in turn means $\hat\psi\rho=\hat\psi'\rho'$. Then the restrictions to $W$ must agree; since the automorphisms $\rho$ and $\rho'$ fix $W$, this means that $\psi=\psi'$, and thus $\hat\psi=\hat\psi'$. It follows that in each nonzero term, we have $\rho=\rho'$, and the sum is $\psi(\epsilon_W(\eta_i))$.
\end{proof}

\begin{corollary} \label{sigma-f-L-j-i-orthogonal}
    The idempotents $\psi\left(\fWjetai\right)$ in \eqref{eq:epsilon-chi-E-f} are orthogonal, that is, for all $\psi,\psi'\in\Gal(W/F)$, $1\le i,i'\le \vf$, and $1\le j,j'\le n_\eta$,
    \[\psi\left(\fWjetai\right)\cdot \psi'\left(f^{(j')}_{W,\eta_{i'}}\right)=\delta_{\psi\psi'}\delta_{ii'}\delta_{jj'} \psi\left(\fWjetai\right).\]
\end{corollary}
\begin{proof}
    Since the idempotents $\psi\left(\fWjetai\right)$ are summands of the idempotents $\psi\left(\epsilon_W(\eta_i)\right)$, and the latter are orthogonal by \cref{sigma-epsilon-L-j-i-orthogonal}, the left hand side is zero whenever $\psi\ne\psi'$ or $i\ne i'$. If $\psi=\psi'$ and $i=i'$, then the statement follows from the definition of  $\fWjetai$ by considering the corresponding matrices in $M_{n_i}(D_{\eta_i})$.
\end{proof}

For $1\le i\le \vf$ and $1\le j\le n_\eta$ fixed, summing these elements over $\Gal(W/F)$ defines an element
\begin{equation} \label{eq:ff-E-def}
    \ffFjetai\colonequals\sum_{\psi\in\Gal(W/F)} \psi\left(\fWjetai\right) \in F[H].
\end{equation}
It follows directly from \cref{sigma-f-L-j-i-orthogonal} that these are orthogonal idempotents in $\Q^F(\G)$. Note that $\ffFjetai$ is a summand of $\epsilonFchi$ by \eqref{eq:epsilon-chi-E-f}.

\begin{proposition} \label{prop:ff-E-j-eta-i}
    For $1\le i\le \vf$, $1\le j\le n_\eta$, the elements $\ffFjetai$ are indecomposable orthogonal idempotents in $\Q^F(\G)$.
    Equivalently, $\ffFjetai \Q^F(\G) \ffFjetai$ is a skew field.
\end{proposition}
\begin{proof}
The proof begins in the same way as in the totally ramified case recalled in \cref{sec:indecomposables}: let $x\in\Q^F(\G)$ be such that $\ffFjetai x\ffFjetai\ne0$. We will show that there exists $y\in\Q^F(\G)$ such that $\ffFjetai y \ffFjetai$ is a left inverse of $\ffFjetai x \ffFjetai$, that is,
\begin{equation}\label{eq:ff-solubility}
	\ffFjetai y \ffFjetai \cdot \ffFjetai x \ffFjetai  = \ffFjetai \gamma^0 \ffFjetai  .
\end{equation}
Let us expand the expression on the left hand side by using the definition of the orthogonal idempotents $\ffFjetai$ and \cref{recall-multiplication}.
\begin{align*}
& \ffFjetai y \ffFjetai \cdot \ffFjetai x \ffFjetai 
= \sum_{\psi\in\Gal(W/F)}\sum_{\substack{k,\ell=0 \\ \vf f\mid k,\ell}}^{p^{n_0}-1} \psi\left(\fWjetai\right) y_k \cdot \psi\left(a_{W,\frac{k}{\vf f}}\right) \psi\left(\fWjetai\right) \cdot \\
&\cdot \psi\left(a_{W,\frac{k}{\vf f}}^{-1}\right) a_{W,\frac{k}{\vf f}} \ptq^{\frac{k}{\vf f}}(x_{\ell}) a_{W,\frac{k}{\vf f}}^{-1}\cdot \psi\left(a_{W,\frac{k+\ell}{\vf f}}\right) \psi\left(\fWjetai\right) \psi\left(a_{W,\frac{k+\ell}{\vf f}}^{-1}\right)\cdot \gamma^{k+\ell}
\end{align*}
The terms on the right hand side are in $\Q^W(\G)$.

{We solve the equation $\psi$-componentwise. First} consider the component of $\psi=\id$. Since $W(\eta)/W_\chi$ is totally ramified, the equation
\begin{align*}
\sum_{\substack{k,\ell=0 \\ \vf f\mid k,\ell}}^{p^{n_0}-1} \fWjetai y_k \cdot a_{W,\frac{k}{\vf f}} \fWjetai \left(a_{W,\frac{k}{\vf f}}\right)^{-1} \cdot \\ \cdot a_{W,\frac{k}{\vf f}} \ptq^{\frac{k}{\vf f}}(x_{\ell}) \left(a_{W,\frac{k}{\vf f}}\right)^{-1}\cdot a_{W,\frac{k+\ell}{\vf f}} \fWjetai \left(a_{W,\frac{k+\ell}{\vf f}}\right)^{-1}\cdot \gamma^{k+\ell} &= \fWjetai \gamma^0 \fWjetai
\end{align*}
admits a solution $y^{\id}=\sum_{k} y^{\id}_{k} \gamma^k \in \Q^W(\G)$ with $y^{\id}_k\in\Q^W(\Gamma_0)[H]$ by \cref{thm:indecomposability-W}. Set $y^\psi_k\colonequals \psi(y^{\id}_k)\in \Q^W(\Gamma_0)[H]$ and $y^\psi\colonequals \sum_{k} y^\psi_k \gamma^k \in \Q^W(\G)$. Then $y^\psi$ is a solution of the $\psi$-part of \eqref{eq:ff-solubility}. Let $y\colonequals \sum_{\psi} y^\psi \in \Q^ F(\G)$: then this is a solution of \eqref{eq:ff-solubility}.
\end{proof}

\begin{remark}
    Note that \cref{prop:ff-E-j-eta-i} is not quite a direct generalisation of \cref{thm:indecomposability-W}. Indeed, if $F(\eta)/F_\chi$ is totally ramified, then $W=F_\chi$, and \cref{thm:indecomposability-W} states that $f^{(j)}_{F,\eta_i}$ is indecomposable in $\Q^F(\G)$, whereas \cref{prop:ff-E-j-eta-i} states that 
    \[\ffFjetai=\sum_{\psi\in\Gal(F_\chi/F)} \psi\left(f^{(j)}_{F_\chi,\eta_i}\right)\]
    is indecomposable in $\Q^F(\G)$. The elements $f^{(j)}_{F,\eta_i}$ and $\ffFjetai$ need not be equal.
\end{remark}

\subsection{Numerical invariants in the general case}
\begin{corollary} \label{invariants-of-DFchi}
	We have $\nFchi = n_\eta \vf$ and $\sFchi = s_\eta w_\chi/\vf = s_\eta (F(\eta):F_\chi)$, that is, the first two assertions in \cref{thm:main} hold.
\end{corollary}
\begin{proof}
	This is the same argument as in \cite[\S6.2]{W}.
    {Combining \eqref{eq:epsilon-chi-E-f} with \eqref{eq:ff-E-def}, we have
    \begin{equation} \label{eq:epsilon-ff}
        \epsilonFchi=\sum_{i=1}^{\vf} \sum_{j=1}^{n_\eta} \ffFjetai.
    \end{equation}
    \Cref{prop:ff-E-j-eta-i} implies that this gives rise to a composition series for $\Q^F(\G)\epsilonFchi$ of length $\vf n_\eta$. 
    
    On the other hand, for $1\le j\le \nFchi$, let $f_{F,\chi}^{(j)}$ be the preimage of the matrix with a single nonzero entry $1$ in the $j$th position in the diagonal, under the isomorphism $\Q^F(\G)\epsilon_{F,\chi}\simeq M_{n_\chi}(D_\chi)$. The idempotent $\epsilonFchi$ also decomposes as the sum of these indecomposable idempotents $f_{F,\chi}^{(j)}$, which yields a composition series of length $\nFchi$. 
    
    All composition series have the same length, and the first assertion follows. Nickel showed that $\sFchi=w_\chi s_\eta n_\eta/\nFchi$ \cite[proof of Corollary~1.13]{NickelConductor}. Substituting the formula for $\nFchi$ together with \cref{item:Feta-Fchi-degree} implies the second assertion.}
\end{proof}

\section{Describing the skew fields} \label{sec:description}
We now proceed to prove \cref{item:main-D}. Let $\chi\in\Irr(\G)$ and let $\eta\mid\res^G_H\chi$ be an irreducible constituent of the restriction. Using indecomposability of the idempotents $\ffFjeta$ established above, we will use the description of $\DWchi$ in the totally ramified case to obtain a description of $\DFchi$.
As we will see in \cref{DFchichi}, the actually hard step is passing from $W$ to $F_\chi$, whereas passing from $F_\chi$ to $F$ is more or less automatic.

\begin{corollary} \label{cor:Dchi-ff}
    For every $\chi\in\Irr(\G)$ and $1\le j\le n_\eta$, we have $\DFchi=\ffFjeta \Q^F(\G) \ffFjeta$.
\end{corollary}
\begin{proof}
    This follows from \cref{prop:ff-E-j-eta-i} and \eqref{eq:epsilon-ff}.
\end{proof}

Recall that in the Wedderburn decomposition \eqref{eq:QFG-Wedderburn} of $\Q^F(\G)$, the character $\chi$ runs through the $\sim_F$-equivalence classes of irreducible characters of $\G$ with open kernel, and $\chi\sim_F\chi'$ if there is a $\sigma\in\Gal( F_\chi/ F)$ such that ${}^\sigma(\res^\G_H \chi)=\res^\G_H \chi'$. 
Heuristically, the centre $\cent(\DFchi)$ is closely connected to $F_\chi$: this is \cite[Lemma~6.8 and Corollary~6.11]{W} in the totally ramified case, and see also \cite[Proposition~1.5]{NickelConductor}.
As $\DFchi$ is cut out by $\ff$-idempotents (\cref{prop:ff-E-j-eta-i}), it is therefore opportune to work with $\ff$-idempotents defined over $F_\chi$ rather than $F$. This can be done by making an appropriate choice for the representative of the $\sim_F$-equivalence class of $\chi$:
\begin{corollary} \label{DFchichi}
	For every $\chi\in\Irr(\G)$, there is a $\chi^*\in\Irr(\G)$ such that $\chi^*\sim_ F\chi$ and $\DFchi = D_{F_{\chi^*},\chi^*}$.
\end{corollary}
\begin{proof}
	Let $1\le j\le n_\eta$ be arbitrary.
    The inclusion $\Q^ F(\G)\subseteq\Q^{ F_\chi}(\G)$ together with \cref{prop:ff-E-j-eta-i} induces an inclusion
	\[\DFchi = \ffFjeta \Q^ F(\G) \ffFjeta \hookrightarrow \ffFjeta \Q^{ F_\chi}(\G) \ffFjeta=\bigoplus_{\chi'\sim_ F \chi} D_{F_{\chi'},\chi'}.\]
	Here the sum is taken over $F_\chi$-equivalence classes of irreducible characters which are $ F$-equivalent to $\chi$. Projecting to the components on the right, we have maps $\DFchi\to D_{F_{\chi'},\chi'}$. A homomorphism of skew fields is either injective or zero, and since the displayed map above is injective, there is a character $\chi^*$ on the right hand side such that $\DFchi\hookrightarrow D_{F_{\chi^*},\chi^*}$. {We have $\sigma\circ\res^\G_H\chi=\res^\G_H\chi^*$ for some $\sigma\in\Gal(F_\chi/F)$, and therefore
    \begin{align*}
        F_{\chi^*}=F\big(\chi^*(h):h\in H\big)=F\big(\sigma(\chi(h)):h\in H\big)=F\big(\chi(h):h\in H\big)=F_\chi.
    \end{align*}
    Similarly, if $\eta^*\mid\res^G_H\chi^*$ is an irreducible constituent, then $\eta^*=\sigma\circ\eta$ for some irreducible constituent $\eta\mid\res^G_H\eta$, and have $F(\eta^*)=F(\eta)$. It follows that $s_\eta$ and $(F(\eta):F_\chi)=w_\chi/\vf$ are the same for $\chi$ and $\chi^*$, and \cref{invariants-of-DFchi} now shows that $\DFchi$ and $D_{F_{\chi^*},\chi^*}$ have the same Schur indices. This forces the injection to be an isomorphism.}
\end{proof}

We are now ready to prove \cref{item:main-D}. More precisely, we will prove the following, which also clarifies the relationship between $\DFchi$ and $\DWchi$.
\begin{corollary} \label{cor:D} We have
	\[\DFchi \simeq \bigoplus_{d=0}^{f-1} \DWchi (\gcpf)^d \simeq \bigoplus_{\substack{\ell=0 \\ \vf \mid \ell}}^{p^{n_0}-1} \tD_\eta (\gcpf)^{\frac{\ell}{\vf }} \simeq \Quot\left(\OO_{D_\eta}[[X,\tauf,\tauf-\id]]\right)\]
	with conjugation by $\gcpf$ acting as $\tauf$ on $\DWchi$ resp. $\tD_\eta$.
\end{corollary}
For brevity, we will write $\AAchi$ resp. $\Achi$ for the first resp. second direct sum in \cref{cor:D}. Note that $(\gcpf)^f=\gamma^{\prime\prime}_{W,\eta}$ by \cref{item:changing-gcp}, and that the element $(\gcpf)^{p^{n_0}/\vf}$ is in the centre of $\tD_\eta$ by \cref{item:multiplication-rule}, so $\AAchi$ and $\Achi$ are well defined. 

\begin{proof}
    We will verify the isomorphisms from right to left. First, observe that the total ring of quotients on the right is a skew field: this follows from \cite[Corollary 2.10(i)]{VenjakobWPT}.
    The isomorphism between the total ring of quotients and $\Achi$ was established in \cite[Proposition 6.9]{W}.

    Since $F(\eta)/W$ is totally ramified, \cref{DchiW-bigoplus} holds for the base field $W$, that is, we have
    \[\DWchi\simeq \bigoplus_{\substack{\ell=0 \\ v^W_\chi \mid \ell}}^{p^{n_0}-1} \tD_\eta \left(\gamma^{\prime\prime}_{W,\eta}
    \right)^{\frac{\ell}{v^W_\chi}},\]
    with $\gamma^{\prime\prime}_{W,\eta}$ acting as $\tau_W$.
    Using the base change statements in \cref{item:change-tau,item:changing-gcp}, this verifies $\AAchi\simeq\Achi$.

    The total ring of quotients on the right (and hence also $\AAchi$ and $\Achi$) has Schur index $(F(\eta):F_\chi) s_\eta=w_\chi/\vf \cdot s_\eta$, see \cite[Corollary~3.3]{W}. \Cref{invariants-of-DFchi} shows that this agrees with the Schur index of $\DFchi$.

    We shall construct an injective ring homomorphism $\Delta: \DFchi\hookrightarrow \AAchi$. Such a homomorphism is necessarily an isomorphism: indeed, we then have $\cent(\AAchi)\subseteq \Delta(\cent(\DFchi)) \subseteq \Delta(\DFchi)\subseteq \AAchi$, with the square of the respective Schur indices being 
    \[\dim_{\cent(\AAchi)} \AAchi=\left(\frac{w_\chi}{\vf}\cdot s_\eta\right)^2=\sFchi^2=\dim_{\Delta(\cent(\DFchi))} \Delta(\DFchi).\] 
    These force $\Delta(\DFchi)=\AAchi$, thereby verifying the isomorphism $\DFchi\simeq\AAchi$.

    By \cref{cor:Dchi-ff,DFchichi}, for any $1\le j\le n_\eta$, we have
    \[\ffFjeta \Q^F(\G) \ffFjeta \simeq \DFchi = D_{F_{\chi^*},\chi^*} \simeq \ff^{(j)}_{F_{\chi^*},\eta} \Q^{F_{\chi^*}}(\G) \ff^{(j)}_{F_{\chi^*},\eta}.\]
    After possibly twisting $\chi$ by a type~W character, that is, a character of $\G$ with trivial restriction to $H$, we may assume that $\chi=\chi^*$.

    Let $x\in\Q^{F_\chi}(\G)$. Unravelling the definition \eqref{eq:ff-E-def} of $\ffFchijeta$ and using the fact that $\Gal(W/F_\chi)$ is cyclic of order $f$ generated by $\tauf|_W$, we have:
    \[\ffFchijeta x \ffFchijeta = \sum_{k,\ell=0}^{f-1} \tauf^k\!\left(\fWjeta\right) x \tauf^\ell\!\left(\fWjeta\right) = \sum_{d=0}^{f-1} \sum_{k=0}^{f-1} \tauf^k\!\left( \fWjeta x \tauf^d\!\left(\fWjeta\right)\right).\]
    Here we used that $x\in\Q^{F_\chi}(\G)$, whence $\tauf(x)=x$. Since each $\tauf^d\!\left(\fWjeta\right)$ is an indecomposable idempotent in $\Q^W(\G)$ by \cref{thm:indecomposability-W}, and $x\in\Q^{F_\chi}(\G)\subseteq \Q^{W}(\G)$, we have 
    \[\fWjeta x \tauf^d\!\left(\fWjeta\right)\in \fWjeta \Q^W(\G) \tauf^d\!\left(\fWjeta\right)\simeq\DWchi.\] 
    Now we can define the image under $\Delta$:
    \begin{align*}
        \Delta\left(\ffFchijeta x \ffFchijeta\right) &\colonequals \sum_{d=0}^{f-1} \fWjeta x \tauf^d\!\left(\fWjeta\right) \cdot \left(\gcpf\right)^d .
    \end{align*}
    This defines a map
    \[\Delta: \DFchi=\ffFchijeta \Q^{F_{\chi}}(\G) \ffFchijeta \to \bigoplus_{d=0}^{f-1} \DWchi \left(\gcpf\right)^d = \AAchi.\]
    We now check that $\Delta$ is a homomorphism. Additivity is clear; we verify multiplicativity. Let $x,y\in \Q^{F_\chi}(\G)$. We first compute the image of the products:
    \begin{align*}
        \Delta\left(\ffFchijeta x \ffFchijeta \cdot \ffFchijeta y \ffFchijeta\right) 
        &= \Delta\left(\sum_{d,k,d',k'=0}^{f-1} \tauf^k\!\left(\fWjeta x \tauf^d\left(\fWjeta\right)\right)\tauf^{k'}\!\left(\fWjeta y \tauf^{d'}\!\left(\fWjeta\right)\right)\right) \\
        \intertext{The orthogonality relation \cref{sigma-f-L-j-i-orthogonal} forces $k'=k+d$:}
        &= \Delta\left(\sum_{d,k,d'=0}^{f-1} \tauf^k\!\left(\fWjeta x \tauf^d\!\left(\fWjeta\right) y \tauf^{d+d'}\!\left(\fWjeta\right)\right)\right) \\
        \intertext{Let $d''\colonequals d+d'$:}
        &= \Delta\left(\sum_{d''=0}^{f-1}\sum_{k=0}^{f-1} \tauf^k\!\left(\sum_{d=0}^{f-1} \fWjeta x \tauf^d\!\left(\fWjeta\right) y \tauf^{d''}\!\left(\fWjeta\right)\right)\right) \\
        &= \sum_{d''=0}^{f-1} \sum_{d=0}^{f-1} \fWjeta x \tauf^d\!\left(\fWjeta\right) y \tauf^{d''}\!\left(\fWjeta\right)\left(\gcpf\right)^{d''}
    \end{align*}
    We also compute the product of the images:
    \begin{align*}
        \Delta\left(\ffFchijeta x \ffFchijeta\right) \cdot \Delta\left(\ffFchijeta y \ffFchijeta\right) 
        &= \sum_{d,d'=0}^{f-1} \fWjeta x \tauf^d\!\left(\fWjeta\right) \left(\gcpf\right)^d \cdot \fWjeta y \tauf^{d'}\!\left(\fWjeta\right) \cdot \left(\gcpf\right)^{d'}\\
        \intertext{Recall that conjugation by $\gcpf$ acts as $\tauf$. Since $y\in\Q^{F_\chi}(\G)$, it is fixed by $\tauf$. Setting $d''\colonequals d+d'$, we obtain:}
        &= \sum_{d''=0}^{f-1} \sum_{d'=0}^{f-1} \fWjeta x \tauf^d\!\left(\fWjeta\right) y \tauf^{d''}\!\left(\fWjeta\right) \cdot \left(\gcpf\right)^{d''},\\
    \end{align*}
    which verifies multiplicativity of $\Delta$.
    
    Since $\Delta$ is a ring homomorphism and its domain $D_\chi$ is a skew field, $\Delta$ is either the zero map or injective. It is clear that $\Delta\ne0$, hence $\Delta$ is an injective homomorphism (and in fact an isomorphism, as explained above).
\end{proof}

The centre of $\DFchi$ can be described in terms of $\gcpf$. This should be compared to Nickel's description of the centre \cite[Proposition~1.5]{NickelConductor}. Let $\Gcpf\colonequals \overline{\langle\gcpf\rangle}$ denote the procyclic group topologically generated by $\gcpf$.
\begin{corollary} \label{centre-with-Gcp}
$\cent\left(\DFchi\right) \simeq \Q^{ F_\chi}\left((\Gcpf)^{w_\chi/\vf}\right)$.
\end{corollary}
\begin{proof}
    This follows directly from \cref{cor:D} and \cite[Proposition 6.9]{W}.
\end{proof}

Finally, we describe maximal orders in $\Q^F(\G)$.
\begin{corollary} \label{cor:max-orders}
Every maximal $\Lambda^{\OO_F}(\Gamma_0)$-order in $\Q^F(\G)$ is isomorphic to one of the form
\[\bigoplus_{\chi\in\Irr(\G)/\sim_{ F}} u_\chi M_{\nFchi} \left( \OO_{D_\eta}[[X; \tauf, \tauf-\id]] \right) u_\chi^{-1} ,\]
where $u_\chi\in \GL_{\nFchi}\left(\Quot\left(\OO_{D_\eta}[[X; \tauf, \tauf-\id]]\right)\right)$.
\end{corollary}
\begin{proof}
    Using \cref{thm:main}, the same proof as in \cite[Corollary 4.11]{W} applies.
\end{proof}

\section{Unramified examples} \label{sec:unramified}
We provide a set of examples in which $F(\eta)/F_\chi$ is a nontrivial unramified extension. Such extensions are not covered by the results of \cite{W}, but our general \cref{thm:main} applies. A special case of the construction below was presented in \cite[Example~2.4.12]{thesis}.

Let $p$ be a Sophie Germain prime, that is, a prime such that $2p+1$ is also prime. It is conjectured that there are infinitely many such primes \cite[\S5.5.5]{Shoup}. Since $2p+1$ is prime, there exists an integer $m$ such that $m^p\equiv 1\bmod{2p+1}$ but $m\not\equiv 1\bmod{2p+1}$. Then there exist distinct integers $1\le a\ne b\le 2p$ such that
\[\{1,2,\ldots,2p\} = \{a,am,\ldots,am^{p-1}\} \sqcup \{b,bm,\ldots,bm^{p-1}\},\]
where $\sqcup$ denotes disjoint union of sets.
Let $S_a\colonequals\sum_{i=0}^{p-1} \zeta_{2p+1}^{am^i}$ and $S_b\colonequals\sum_{i=0}^{p-1} \zeta_{2p+1}^{bm^i}$. Since the sum of primitive roots of unity of order $2p+1$ is $-1$, we have $S_a+S_b=-1$. Note that $S_a$ and $S_b$ are fixed by the automorphism of $\QQ_p(\zeta_{2p+1})$ given by $\zeta_{2p+1}\mapsto\zeta_{2p+1}^m$. This automorphism generates a subgroup of order $p$ in the cyclic group $\Gal(\QQ_p(\zeta_{2p+1})/\QQ_p)$ of order $2p$. Therefore by Galois theory, $\QQ_p(S_a)=\QQ_p(S_b)$ is a quadratic extension of $\QQ_p$.

Let $H$ be the cyclic group of order $2p+1$, and let $C_p=\langle g\rangle$ act on $H$ as the $m$th power map, that is, $ghg^{-1}=h^m$ for all $h\in H$. Let $\G\colonequals H\rtimes \Gamma$ where $\Gamma$ acts on $H$ via the natural projection $\Gamma\simeq\ZZ_p\twoheadrightarrow\ZZ_p/p\ZZ_p\simeq C_p$. Let $\chi\in\Irr(\G)$ be an irreducible character of degree $p$. Let $\eta\mid\res^G_H\chi$ be an irreducible constituent of its restriction to $H$. By definition of the group action, we have $w_\chi=p$.
By Clifford theory, see \cite[(3)]{NickelConductor}, we have the following for all $h\in H$:
\begin{equation} \label{eq:ex-res}
    \left(\res^\G_H\chi\right)(h)=\sum_{i=0}^{p-1} {}^{g^i}\eta(h)=\sum_{i=0}^{p-1} \eta\left(g^i h g^{-i}\right)=\sum_{i=0}^{p-1} \eta\left(h^{m^i}\right)=\sum_{i=0}^{p-1} \eta(h)^{m^i}.
\end{equation}
Since $\eta$ is a nontrivial character of the cyclic group of order $2p+1$, we have $\QQ_p(\eta)=\QQ_p(\zeta_{2p+1})$. Moreover, \eqref{eq:ex-res} combined with the discussion above shows that $\QQ_{p,\chi}$ is a quadratic extension of $\QQ_p$. The extension $\QQ_p(\zeta_{2p+1})/\QQ_p$ is unramified of degree $2p$, hence $\QQ_p(\eta)$ is an unramified extension of $\QQ_{p,\chi}$ of degree $p$. More generally, if $F/\QQ_p$ is a finite extension whose inertia degree is coprime to $p$, then $F(\eta)/F_\chi$ is also unramified of degree $p$.

\section{Dyadic Wedderburn decomposition} \label{sec:2adic}
Until now, we always assumed $p$ to be an odd prime. Now we briefly describe the extent to which the algebraic results above carry over to the $2$-adic setting, and highlight the obstacles to a full generalisation.
We emphasise the algebraic nature of these results: when it comes to arithmetic applications in Iwasawa theory, the case of $p=2$ presents significant challenges.

For this very reason, the works \cite{TEIT-II,NickelConductor} cited above, in which idempotents of $\Q^F(\G)$ are studied, also assume $p\ne2$.
Indeed, Ritter--Weiss assume the underlying prime to be odd, although the algebraic part (\S2) of their paper does not use this assumption. Nickel builds on Ritter--Weiss's work, and hence also assumes $p\ne2$, but once again, the results contained in the relevant part (\S1) of his article remain valid for $p=2$.

\subsection{Trivial Schur index}
The author's previous work \cite{W} on the totally ramified case also assumed $p=2$. This assumption is actually used, namely in the construction of extensions of Galois automorphisms to skew fields as described in \cref{extending-tau}. 
As explained at the end of \S4.2 of \cite{
W}, this relies on Witt's result \cite[Satz~10]{Witt} that the Schur index $s_\eta$ divides $p-1$, which is false for $p=2$. In this case, it is known that $s_\eta\in\{1,2\}$, and we have $s_\eta=1$ if one of the following conditions holds:
\begin{enumerate}
	\item the base field $F$ contains a primitive 4th root of unity \cite[Satz~11]{Witt},
	\item all $2$-Sylow subgroups of $H$ are abelian \cite{Yamada78}.
\end{enumerate}
In fact, Yamada \cite{Yamada2,Yamada2B} gave an explicit description of the cases when $s_\eta=2$.
Note that if $s_\eta=1$, then we have $D_\eta=F(\eta)$, so there is no need to extend automorphisms in $\Gal(F(\eta)/F_\chi)$ to $D_\eta$.
We can show that our results remain valid for each component satisfying $s_\eta=1$:
\begin{proposition}
	Let $p=2$, and let $F$ be a finite extension of $\QQ_2$. If $\chi\in\Irr(\G)$, and $\eta\mid\res^\G_H\chi$ is an irreducible constituent of its restriction with Schur index $s_\eta=1$, then 
	\begin{propositionlist}
		\item the size of the matrix ring is $\nFchi = \eta(1) \vf$,
		\item the Schur index of $\DFchi$ is $\sFchi = (F(\eta):F_\chi)$,
		\item the skew field part of the $\chi$-component is $\DFchi \simeq \Quot\left(\OO_{F(\eta)}[[X; \tauf, \tauf-\id]]\right)$.
	\end{propositionlist}
\end{proposition}
\begin{proof}
	First suppose that $F(\eta)/F_\chi$ is totally ramified. Let us trace through the proof of \cite[Theorem~6.1]{W}, in order to check that $\fetaj$ is an indecomposable idempotent in $\Q^F(\G)$. The reformulation of this claim in terms of linear algebra carries through: for this, one needs the automorphism $\pt$, which is present since $\tau$ already acts on $D_\eta=F(\eta)$ without any additional argument. Next, one considers the algebra
	\[\bigoplus_{i=1}^{p^{n_0}/\vf} \Q^{F(\eta)}(\Gamma_0)\left(\gamma_0^{\vf/p^{n_0}}\right)^{i}=\bigoplus_{j=0}^{w_\chi/\vf-1}(\Gamma^{w_\chi})\gamma^{\vf j} = \left(\Q^{F(\eta)}(\Gamma^{w_\chi})\big/\Q^{F_\chi}(\Gamma^{w_\chi}), \tau, \gamma^{w_\chi}\right),\]
	and the claimed indecomposability is equivalent to this algebra being a skew field.
	Since $s_\eta=1$, one can avoid using Lemma~6.3, Corollary~6.4, and Lemma~6.5 of \cite{W}. The claimed indecomposability then follows from Lemma~6.6: this relies on \cite[Corollary~2.11]{EN}, which is valid for $p=2$. One also needs the decomposition $\OO_{F_\chi}^\times\simeq\mu_{\#\overline{F_\chi}-1}\times U^1_{F_\chi}$, where $\overline{F_\chi}$ is the residue field of $F_\chi$. Observe that the emergence of $2$-power torsion elements in the group of $1$-units $U^1_{F_\chi}$ does not affect the validity of the argument. Therefore $\fetaj$ is indecomposable, as claimed.
	
	Now let $F(\eta)/F_\chi$ have arbitrary ramification. The arguments in \cref{sec:indecomposables-QF,sec:description} don't rely on the $p\ne2$ assumption beyond the fact that $\tau$ extends to $D_\eta$. Therefore the proofs of \cref{invariants-of-DFchi,cor:D} apply.
\end{proof}

\subsection{Nontrivial Schur index}
As we shall see below, there are indeed examples in which $s_\eta=2$ and \cref{extending-tau} fails. 
Let us explain the difficulties arising from this. 
The proof of \cref{recall-multiplication} involves comparing two actions on each side of
\begin{equation} \label{eq:FHeta}
	F[H]\epsilon_F(\eta) \simeq M_{n_\eta}(D_\eta).
\end{equation}
On the left hand side, we always have the action of conjugation by $\gamma^{\vf}$. On the right hand side, $\tauf$ acts coefficientwise, provided that it extends to $D_\eta$.
Even without such an extension, we can get an automorphism on the right hand side by pushing forward the conjugation action: this defines an automorphism $\delta_\gamma\in\Aut_{F_\chi}(D_\eta)$ of order $w_\chi/\vf$. However, this may only act on the matrix ring rather than the skew field $D_\eta$. In addition, even if $\delta_\gamma$ does act on $D_\eta$ (for instance when $n_\eta=1$), the resulting action need not preserve the fields $F(\eta)(\omega)$ and $F(\eta)(\pi)$, which also means that one has much less control of the images $\delta_\gamma(\omega)$ and $\delta_\gamma(\pi)$. Knowledge of these, however, is needed in the proof of \cref{thm:indecomposability-W} in the totally ramified case, see \cite[Corollary~6.4]{W}.

\subsection{Example}

We conclude this short excursion to the $2$-adic world by an explicit counterexample in which the automorphism cannot be extended to the skew field as in \cref{extending-tau}.

Consider the extension $\QQ_2(\sqrt2)/\QQ_2$: this is a totally ramified cyclic $2$-extension.
The only nontrivial Galois conjugation is given by $\tau(\sqrt2)=-\sqrt2$, and $\sqrt2$ is a uniformiser in $\QQ_2(\sqrt2)$.
Let $\omega$ be a root of unity of order $2^2-1=3$, and let $\sigma$ be the Frobenius, that is, $\sigma(\omega)\colonequals\omega^2$.
Consider the cyclic algebra 
\[D\colonequals\big(\QQ_2(\sqrt2,\omega)/\QQ_2(\sqrt2),\sigma,\sqrt2\big)=\QQ_2(\sqrt2,\omega)\oplus\QQ_2(\sqrt2,\omega)\sqrt[4]{2},\]
where $\pi\colonequals\sqrt[4]{2}$ acts as $\pi\omega\pi^{-1}=\sigma(\omega)$.
This is a skew field by \cite[Corollary~31.6]{MO}: indeed, $\sqrt2$ is a uniformiser, so $[\QQ_2(\sqrt2,\omega):\QQ_2(\sqrt2)]=2$ is coprime to $v_{\QQ_2(\sqrt2)}(\sqrt2)=1$. We claim that there is no extension of $\tau$ to $D$ as in \cref{extending-tau}. In fact, there is no extension to $\QQ_2(\sqrt[4]{2})$. Indeed, if there were, then by the explanation in \cite[Lemma 2.4]{W}, the quotient $\tau(\sqrt2)/\sqrt2=-1$ would admit a square root in $\QQ_2(\sqrt[4]{2})$. But $-1$ is not a square in $\QQ_2(\sqrt[4]{2})$: indeed, this is a totally ramified extension of $\QQ_2$, whereas $\QQ_2(\sqrt{-1})/\QQ_2$ is unramified.

We show that this example actually occurs in the study of Wedderburn decompositions of $2$-adic Lie groups. Let $F\colonequals\QQ_2$, and let 
\[H\colonequals Q_{16}=\langle a,b\mid a^8=b^4=1, a^4=b^2, bab^{-1}=a^{-1}\rangle\] 
be the generalised quaternion group of order 16. Let $C_2$ be the cyclic group of order $2$. This acts on $H$ as follows: let $c$ be the generator of $C_2$, and let $c$ act by $a\mapsto a^3$ and $b\mapsto b$. This defines a semidirect product $C_2\ltimes H$; its GAP SmallGroup ID is $(32,44)$. Letting $\Gamma$ act on $H$ via $\Gamma\twoheadrightarrow C_2$, we obtain $\G=\Gamma\ltimes H$. Let $\eta$ be either one of the two faithful two-dimensional characters of Schur index $2$ of $H$: then $\QQ_2(\eta)=\QQ_2(\sqrt2)$. Then $\chi\colonequals\ind^\G_H\eta$ is an irreducible degree~$4$ character factoring through $C_2\ltimes H$ with restriction $\res^\G_H\chi=\eta+{}^c\eta$, so $\QQ_{2,\chi}=\QQ_2$.

We give a partial description of a Wedderburn isomorphism \eqref{eq:FHeta}. It is clear from the construction of $\G$ that $w_\chi=2$, hence $\vf=1$. Since $s_\eta n_\eta=\eta(1)=2$, we have $n_\eta=1$. A quick computation shows that $\epsilon\colonequals\epsilon_{\QQ_2}(\eta)=\frac12(1-a^4)$, and therefore $a^4\epsilon=-\epsilon$ in $\QQ_2[H]\epsilon$. We now identify elements in $\QQ_2[H]\epsilon$ corresponding to $\sqrt2$ and $\omega$. Since $\sqrt2$ is central in $D_\eta$, there are exactly two square roots of $2$ in $D_\eta$, each corresponding to one of the two square roots of $2$ in $\QQ_2[H]\epsilon$, namely $\pm(a+a^7)\epsilon$. Let us assume $(a+a^7)\epsilon\leftrightarrow\sqrt2$. Since $\QQ_2(\sqrt2)$ contains no primitive third roots of unity, there are infinitely many such elements, which all differ by conjugation. One of them is $\tilde \omega\colonequals \frac12(-1+a^2+b+a^2b)\epsilon$; this is known as a Hurwitz unit, see \cite[\S11]{Voight}.
Pre- or postcomposing a Wedderburn isomorphism with conjugation produces another Wedderburn isomorphism, therefore a choice can be made, and we may assume $\tilde \omega\leftrightarrow\omega$.
We conclude that there exists a Wedderburn isomorphism under which
\[\QQ_2[H]\epsilon\simeq D_\eta, \quad (a+a^7)\epsilon\leftrightarrow\sqrt2, \quad \tilde \omega\leftrightarrow\omega.\]
An explicit description of the element corresponding to $\sqrt[4]{2}$ appears to be significantly more elusive.

Note that $\delta_\gamma$ does not preserve the inertia field $\QQ_2(\sqrt2,\omega)$, that is $\delta_\gamma(\QQ_2(\sqrt2,\omega))\ne\QQ_2(\sqrt2,\omega)$. Indeed, since $\delta_\gamma$ is an automorphism of $D_\eta$, it sends primitive third roots of unity to primitive third roots of unity. If it would preserve $\QQ_2(\sqrt2,\omega)$, conjugation by $\gamma$ would send $\tilde \omega$ either to itself or to $\tilde \omega^2=\frac12(-1-a^2-b-a^2b)$. However, $\gamma \tilde \omega\gamma^{-1}=\frac12(-1-a^2+b-a^2b)$ is neither.

\printbibliography
\end{document}